\documentclass{amsart}
\usepackage{dsfont}
\usepackage[usenames,dvipsnames,svgnames,table]{xcolor}
\usepackage[utf8]{inputenc}
\usepackage[T1]{fontenc}
\usepackage{ tipa }
\usepackage{helvet}
\usepackage{color}
\usepackage{mathrsfs}  
\usepackage{stmaryrd}  
\usepackage{amsmath,amsxtra,amsthm,amssymb,xr,enumerate,fullpage,hyperref,url}
\usepackage[all]{xy}
\usepackage[bbgreekl]{mathbbol}
\usepackage{bbm}

\usepackage{mathtools}
\DeclarePairedDelimiter{\ceil}{\lceil}{\rceil}

\DeclareMathSymbol{\invques}{\mathord}{operators}{`>}
\DeclareUnicodeCharacter{00BF}{\tmquestiondown}
\DeclareRobustCommand{\tmquestiondown}{%
  \ifmmode\invques\else\textquestiondown\fi
}

\usepackage{verbatim}

\newtheorem{thm}{Theorem}[section]
\newtheorem{lem}[thm]{Lemma}

\newtheorem{prop}[thm]{Proposition}
\newtheorem{cor}[thm]{Corollary}

\newtheorem{remark}[thm]{Remark}

\newcommand{\Qp}{\mathbb{Q}_p}
\newcommand{\Zp}{\mathbb{Z}_p}

\newcommand{\ord}{\operatorname{ord}}

\newcommand{\bthm}{\begin{thm}}
\newcommand{\ethm}{\end{thm}}
\newcommand{\bprop}{\begin{prop}}
\newcommand{\eprop}{\end{prop}}
\newcommand{\blem}{\begin{lem}}
\newcommand{\elem}{\end{lem}}

\usepackage{amsthm}

\newcommand{\bcor}{\begin{cor}}
\newcommand{\ecor}{\end{cor}}
\newcommand{\beq}{\begin{equation}}
\newcommand{\eeq}{\end{equation}}
\newcommand{\bpr}{\begin{proof}}
\newcommand{\epr}{\end{proof}}
\newcommand{\brem}{\begin{remark}}
\newcommand{\erem}{\end{remark}}
\def\bal#1\eal{\begin{align*}#1\end{align*}}
\def\bga#1\ega{\begin{gather*}#1\end{gather*}}

\makeatletter
\let\origsubsection\subsection
\renewcommand\subsection{\@ifstar{\starsubsection}{\nostarsubsection}}

\newcommand\nostarsubsection[2][\relax]{%
  \subsectionprelude%
  \ifx\relax#1\origsubsection{#2}\else\origsubsection[#1]{#2}\fi%
  \subsectionpostlude}

\newcommand\starsubsection[2][\relax]{%
  \subsectionprelude%
  \ifx\relax#1\origsubsection*{#2}\else\origsubsection*[#1]{#2}\fi%
  \subsectionpostlude}

\newcommand\subsectionprelude{%
  \vspace{.75\parskip}% OR ANY OTHER DESIRED VALUE
}

\newcommand\subsectionpostlude{%
  \vspace{-\parskip}% OR ANY OTHER DESIRED VALUE
}
\makeatother

\usepackage{hyperref}

\begin{document}

\title{Explicit uniformizers for certain totally ramified extensions of the field of $p$-adic numbers
}

%\titlerunning{Explicit uniformizers for certain totally ramified extensions of $\Qp$}        % if too long for running head

\begin{abstract}
Let $p$ be an odd prime number. We construct explicit uniformizers for the totally ramified extension $\Qp(\zeta_{p^2},\sqrt[p]{p})$ of the field of $p$-adic numbers $\Qp$, where $\zeta_{p^2}$ is a primitive $p^2$-th root of unity.
\end{abstract}

\author{Hugues Bellemare}
\address{Hugues Bellemare\newline 
Department of Mathematics and Statistics\\McGill University\\
805 Sherbrooke Street West\\
Montréal, QC\\
Canada H3A 0B9 }
\email{hugues.bellemare@mail.mcgill.ca}

\author{Antonio Lei}
\address{Antonio Lei\newline
D\'epartement de Math\'ematiques et de Statistique\\
Universit\'e Laval, Pavillion Alexandre-Vachon\\
1045 Avenue de la M\'edecine\\
Qu\'ebec, QC\\
Canada G1V 0A6}
\email{antonio.lei@mat.ulaval.ca}


\maketitle
\section{Introduction}

Let $p$ be a prime number and $m\ge1$ an integer. Let $\zeta_{p^m}$ be a primitive $p^m$-th root of unity. It is a well-known fact that  $\Qp(\zeta_{p^m})/\Qp$ is a totally ramified extension and that $\zeta_{p^m}-1$ is a uniformizer for $\Qp(\zeta_{p^m})$. Similarly, if $n\ge1$ is an integer, then $\Qp(\sqrt[p^n]{p})/\Qp$ is also a totally ramified extension and $\sqrt[p^n]{p}$ is a uniformizer for this extension. The compositum $K_{m,n}:=\Qp(\zeta_{p^m},\sqrt[p^n]{p})$ is {a totally ramified extension of $\Qp$ of degree $p^{m+n-1}(p-1)$}. It is thus natural to search for an explicit uniformizer for this field.

Indeed, as explained by Viviani in \cite{viv}, explicit uniformizers in $K_{m,n}$ allow us to compute ramification groups of the extension. We expect that a norm-compatible system of uniformizers as $m$ and $n$ vary would have applications in non-commutative Iwasawa Theory as well. For example, in the case of elliptic curves with supersingular reduction at $p$ with $p>3$, Kobayashi \cite{kobayashi03} has constructed a  system of local points on an elliptic curve defined over $\Qp(\zeta_{p^m})$ for $m\ge1$ using the uniformizers $\zeta_{p^m}-1$. These points have led to the definition of plus and minus Selmer groups, which have played an important role in the study of supersingular Iwasawa Theory in recent years. Kim has generalized Kobayashi's construction to the $\Zp^2$-extensions over an imaginary quadratic field where $p$ splits in \cite{kimZp2} as well as to abelian varieties in \cite{kimAV}. A system of uniformizers of $K_{m,n}$ could potentially allow us to define the appropriate analogues of Kobayashi's Selmer groups over these fields.

In \cite[Lemma~6.4]{viv}, Viviani has showed that 
\begin{equation}\label{eq:viv}
    \frac{1-\zeta_p}{\sqrt[p]{p} \cdots \sqrt[p^n]{p}}
\end{equation}
is a uniformizer of $K_{1,n}$. Indeed, if  $\ord_p$ denotes the $p$-adic valuation on $\overline{\Qp}$ normalized by $\ord_p(p)=1$, then
\[
\ord_p\left(\frac{1-\zeta_p}{\sqrt[p]{p}\cdots \sqrt[p^n]{p}}\right)=\frac{1}{p-1}-\sum_{i=1}^n\frac{1}{p^i}=\frac{1}{p^n(p-1)}=\frac{1}{[K_{1,n}:\Qp]}.
\]
Thus, the expression in \eqref{eq:viv} is indeed a uniformizer for $K_{1,n}$. A naive attempt to generalize this construction to $K_{m,n}$ with $m\ge2$ does not result in a uniformizer. The reason is that if we multiply powers of $1-\zeta_{p^m}$ and $\sqrt[p^n]p$, the $p$-adic valuation of such a product will be a linear combination of $\frac{1}{p^{m-1}(p-1)}$ and $\frac{1}{p^n}$, which does not give $\frac{1}{p^{m+n-1}(p-1)}=\frac{1}{[K_{m,n}:\Qp]}$ unless $m=1$ or $n=0$.

In an online discussion on the website StackExchange \cite{mercio} (\url{https://math.stackexchange.com/questions/954731}), the user "Mercio" has suggested a strategy to find explicit uniformizers for $K_{2,1}$. The crux of the construction is to work with the minimal polynomial of $\pi:=\zeta_{p^2}-1$, which gives the equation
\[
\frac{-p}{\pi^{\phi(p^2)}}=1+\sum_{i=1}^{\phi(p^2)-1}a_i\pi^{i-\phi(p^2)},
\]
where $\phi$ is the Euler totient function and $a_i$ are integers divisible by $p$. We write $O(\pi^k)$ to represent an element in $\pi^k\Zp[\pi]$. If we obtain from the  equation above
\[
w^p=c\pi+O(\pi^2)
\]
for certain $c,w\in K_{2,1}$ with $\ord_p(c)=0$, then $\ord_p(w)=\frac{\ord_p(\pi)}{p}=\frac{1}{[K_{2,1}:\Qp]}$. Thus, $w$ would be a uniformizer of $K_{2,1}$. In the case $p=3$, Mercio illustrated this strategy by showing that
\[
\frac{-3}{\pi^{\phi(3^2)}}=1+2\pi^3+\pi^4+O(\pi^5),
\]
where $\pi=\zeta_{3^2}-1$. On setting ${v=\frac{-\sqrt[3]{3}}{(-2-\pi)\pi^2}-1}$, we get
\[
v^3=\pi^4+O(\pi^5).
\]
Dividing both sides by $\pi^3$ results in
\[
w^3=\pi+O(\pi^2),
\]
where $w=v/\pi$. It thus gives a uniformizer.

In this article, we expand Mercio's idea to give a general algorithm that gives us an explicit uniformizer of $K_{2,1}$ for all odd prime $p$. The structure of this article is as follows. We first prove a number of general results on the $p$-adic valuations of certain binomial coefficients in Section~\ref{sec:pre}. We then carry out our construction of explicit uniformizers in Section~\ref{sec:uniformizer}. At the end of the article, we explain why our method does not work for more general $K_{m,n}$.

\subsection*{Acknowledgements}
This work was carried out during two NSERC summer research internships undertaken by the first named author at Universit\'e Laval in 2018 and 2019. The second named author's research is supported by the NSERC Discovery Grants Program 05710. We would like to thank Hugo Chapdelaine, Daniel Delbourgo, C\'edric Dion, Byoung Du Kim and Antoine Poulin for helpful and interesting discussions during the preparation of this article. Finally, we thank the anonymous referee for valuable comments and suggestions on an earlier version of the article.

\section{Notation and preliminary results}\label{sec:pre}
For the rest of this article, $m\ge2$ is a fixed integer and  $\pi$ denotes the uniformizer $\zeta_{p^m}-1$ of $\Qp(\zeta_{p^m})$. We will only apply our results to construct explicit uniformizers when $m=2$. By working with a general $m\ge 2$, we will be able to explain why our method does not extend to the case $m>2$ (see  Remark~\ref{rk:fail} below).

The minimal polynomial of $\pi$ over $\Qp$ is given by
\begin{align*}
    f(x) &= \sum_{k=0}^{p-1}(x+1)^{kp^{m-1}}\\
    &=\sum_{\ell=0}^{\phi(p^m)} \sum_{k=\ceil[\big]{\frac{\ell}{p^{m-1}}}}^{p-1} \binom{kp^{m-1}}{\ell} x^\ell \\ 
&= p + \sum_{\ell=1}^{\phi(p^m)-1} \sum_{k=\ceil[\big]{\frac{\ell}{p^{m-1}}}}^{p-1} \binom{kp^{m-1}}{{\ell}} x^\ell + x^{\phi(p^m)} \\
&= p + \sum_{\ell=1}^{\phi(p^m)-1} a_\ell x^\ell + x^{\phi(p^m)},
\end{align*}
where { $\ceil{x}$ is the ceiling function on the set of real numbers (sending $x$ to the smallest integer $\ge x$) and} $a_\ell=\sum_{k=\ceil[\big]{\frac{\ell}{p^{m-1}}}}^{p-1} \binom{kp^{m-1}}{{\ell}}$. Note that $p|a_\ell$ for all $\ell$ since $f$ is an Eisenstein polynomial. We have
\begin{equation}
    \frac{-p}{\pi^{\phi(p^m)}} = 1+\sum_{\ell=1}^{\phi(p^m)-1} a_\ell\pi^{\ell-\phi(p^m)}.
\label{eq:expand}
\end{equation}
Let $v$ be the normalized $\pi$-adic valuation. In particular $v(p)=\phi(p^m)$ and $v(a_\ell)\ge \phi(p^m)$ for all $\ell$.

We shall be interested in understanding \eqref{eq:expand} modulo $\pi^{d}$, where $d$ is defined by 
\[
d:=(p-2)p^{m-1}+p^{m-2}+1.
\]

\blem\label{lem:largeterms}
If $\ell\ge d$, then $a_\ell \pi^{\ell-\phi(p^m)}=O(\pi^d)$.
\elem
\begin{proof}
This follows from the fact that $p|a_\ell$, which gives $$v(a_\ell\pi^{\ell-\phi(p^m)})=v(a_\ell)+\ell-\phi(p^m)\ge d.$$
\end{proof}

In particular, we deduce from Lemma~\ref{lem:largeterms} that \eqref{eq:expand} implies
\[
\frac{-p}{\pi^{\phi(p^m)}} = 1+\sum_{\ell=1}^{d-1} a_\ell\pi^{\ell-\phi(p^m)}+O(\pi^d).
\]
We shall now study the terms $a_\ell\pi^{\ell-\phi(p^m)}$ for $1\le \ell\le d-1$. We separate $\ell$ in the following cases:
\begin{itemize}
    \item[(i)] $\ell<(p-2)p^{m-1}$ and $\ord_p(\ell)<m-1$;
    \item[(ii)] $\ell=tp^{m-1}$, where $t=1,2,\ldots, p-2$;
    \item[(iii)] $(p-2)p^{m-1}< \ell <d-1$;
    \item[(iv)] $\ell=d-1$.
\end{itemize}

The following lemma will help us to study case (i).
\begin{lem}\label{lem:lwithsmallord}
Suppose that $1\le \ell< (p-2)p^{m-1} $ with $\ord_p(\ell)<m-1$. Then $\ord_p(a_\ell)\ge 2$.
\end{lem}
\begin{proof}
Recall that $a_\ell=\sum_{k=\ceil[\big]{\frac{\ell}{p^{m-1}}}}^{p-1} \binom{kp^{m-1}}{\ell}$. For each $k$, we have
\begin{equation}\label{eq:expandbinom}
    \binom{kp^{m-1}}{\ell}=\frac{kp^{m-1}}{\ell}\cdot \prod_{j=1}^{\ell-1}\frac{kp^{m-1}-j}{j}=\frac{kp^{m-1}}{\ell}\cdot \prod_{j=1}^{\ell-1}\left(\frac{kp^{m-1}}{j}-1\right).
\end{equation}
Since $j\le\ell -1\le (p-2)p^{m-1}$ in the product, we have $\ord_p(j)\le m-1$.  In particular, the product above is in $\Zp$ and thus 
\[
\binom{kp^{m-1}}{\ell}\in \frac{p^{m-1}}{\ell}\Zp.
\]
In the case where $\ord_p(\ell)\le m-3$, the binomial coefficient above is in $p^2\Zp$ for all $k$. In particular, $a_\ell\in p^2\Zp$.

Let us now consider the case $\ord_p(\ell)=m-2$. Let us write $\ell =rp^{m-2}$, where $r$ is an integer coprime to $p$. Then we deduce from \eqref{eq:expandbinom} that
\begin{align*}
 \binom{kp^{m-1}}{\ell}&=\frac{kp}{r}\cdot \prod_{j=1}^{\ell-1}\left(\frac{kp^{m-1}}{j}-1\right)\\
 &=\frac{kp}{r}\cdot \prod_{\ord_p(j)\ge m-2}\left(\frac{kp^{m-1}}{j}-1\right)\prod_{\ord_p(j)<m-2}\left(-1\right)\mod p^2\Zp.
\end{align*}
Furthermore, all the terms $\frac{kp^{m-1}}{j}$ where $\ord_p(j)=m-2$ will be absorbed in $p^2\Zp$ when multiplied by $\frac{kp}{r}$. Thus,
\[
\binom{kp^{m-1}}{\ell}=\frac{kp}{r}\cdot \prod_{\ord_p(j)= m-1}\left(\frac{kp^{m-1}}{j}-1\right)\prod_{\ord_p(j)\le m-2}\left(-1\right)\mod p^2\Zp.
\]
Thus, if we write $s=\#\{j:1\le j\le \ell-1, \ord_p(j)=m-1\}$, then the product above becomes
\[
    \binom{kp^{m-1}}{\ell}=\frac{kp}{r}\cdot \prod_{j=1}^s\left(\frac{k}{j}-1\right)\left(-1\right)^{\ell-s-1}\mod p^2\Zp.
\]

Notice that 
\[
\ceil[\bigg]{\frac{\ell}{p^{m-1}}}=\ceil[\bigg]{\frac{r}{p}}=s+1.
\]
We can therefore deduce that
\begin{align*}
    a_\ell&=\sum_{k=s+1}^{p-1}\frac{kp}{r}\cdot \prod_{j=1}^s\left(\frac{k}{j}-1\right)\left(-1\right)^{\ell-s-1}\mod p^2\Zp\\
    &=\frac{\left(-1\right)^{\ell-s-1}p}{r}\sum_{k=s+1}^{p-1}k\prod_{j=1}^s\left(\frac{k}{j}-1\right)\mod p^2\Zp\\
    &=\frac{\left(-1\right)^{\ell-s-1}p}{r}\sum_{k=s+1}^{p-1}(s+1)\prod_{j=1}^{s+1}\frac{k-j+1}{j}\mod p^2\Zp\\
    &=\frac{\left(-1\right)^{\ell-s-1}p(s+1)}{r}\sum_{k=s+1}^{p-1}\binom{k}{s+1}\mod p^2\Zp\\
    &=\frac{\left(-1\right)^{\ell-s-1}p(s+1)}{r}\binom{p}{s+2}\mod p^2\Zp
\end{align*}
by the hockey stick identity. But our upper bound on $\ell$ gives $s+2<p$, which tells us that $\binom{p}{s+2}\equiv 0\mod p$. Since $p\nmid r$, the result now follows.
\end{proof}

\begin{cor}\label{cor:killlwithsmallord}
Suppose that $1\le \ell< (p-2)p^{m-1} $ with $\ord_p(\ell)<m-1$. Then $a_\ell\pi^{\ell-\phi(p^m)}=O(\pi^d)$.
\end{cor}
\begin{proof}
Lemma~\ref{lem:lwithsmallord} tells us that $v(a_\ell)\ge 2\phi(p^m)$.
Thus, $$v(a_\ell\pi^{\ell-\phi(p^m)})> \phi(p^m)>d$$
and the result follows.
\end{proof}

We now study the case (ii).

\begin{lem}
Let $\ell=tp^{m-1}$, where $t=1,2,\ldots, p-2$. Then
\[
a_{\ell}\in \frac{(-1)^t}{t+1}p+p^2\Zp.
\]
\end{lem}
\begin{proof}
As in the proof of Lemma~\ref{lem:lwithsmallord}, we consider 
\[
\binom{kp^{m-1}}{tp^{m-1}}=\frac{kp^{m-1}}{tp^{m-1}}\cdot \prod_{j=1}^{tp^{m-1}-1}\frac{kp^{m-1}-j}{j}=\frac{k}{t}\cdot \prod_{j=1}^{tp^{m-1}-1}\left(\frac{kp^{m-1}}{j}-1\right).
\]
In order to understand this modulo $p^2$, we split the product into:
\begin{align*}
    &\ \frac{k}{t}\prod_{\ord_p(j)=m-1}\left(\frac{kp^{m-1}}{j}-1\right) \prod_{\ord_p(j)=m-2}\left(\frac{kp^{m-1}}{j}-1\right)&\\
    &\   \prod_{\ord_p(j)<m-2}\left(\frac{kp^{m-1}}{j}-1\right)&\\
    =&\ \frac{k}{t}\prod_{j=1}^{t-1}\frac{k-j}{j}\prod_{1\le j\le tp-1,p\nmid j}\left(\frac{kp}{j}-1\right)(-1)^{\ell-tp} \mod p^2\Zp\\
    =&\ \frac{k}{t}\prod_{j=1}^{t-1}\frac{k-j}{j}\left(1-{kpC}\right) \mod p^2\Zp\\
    =&\ \binom{k}{t}(1-{kpC}) \mod p^2\Zp,
\end{align*}
where $C=\sum_{1\le j\le tp-1,p\nmid j}\frac{1}{j}$.
{We} have
\begin{align*}
    C &= \sum_{j=1,p \nmid j}^{tp-1}\frac{1}{j} = \sum_{i=0}^{t-1}\sum_{j=1}^{p-1}\frac{1}{ip+j} = \sum_{i=0}^{t-1}\sum_{j=1}^{p-1}\frac{1}{j}\mod p\Zp\\
    &= \sum_{i=0}^{t-1}\sum_{j=1}^{p-1}j\mod p\Zp \\ &= t\frac{(p-1)p}{2}=0\mod p\Zp.
\end{align*}
Thus, using the hockey stick identity again, we deduce that
\begin{align*}
    a_\ell&=\sum_{k=t}^{p-1}\binom{k}{t}\mod p^2\Zp\\
    &=\binom{p}{t+1}\mod p^2\Zp\\
    &=p\frac{(p-1)\dots(p-t)}{(t+1)\cdot t \cdot \ldots \cdot 1}\mod p^2\Zp\\
    &= \frac{(-1)^t p}{t+1}\mod p^2\Zp,
\end{align*}
as required.
\end{proof}

\begin{cor}\label{cor:termslargeval}
Let $\ell=tp^{m-1}$, where $t=1,2,\ldots, p-2$. Then
\[
a_{\ell}\pi^{\ell-\phi(p^m)}= \frac{(-1)^t}{t+1}p\pi^{(t-p+1)p^{m-1}}+O(\pi^d).
\]
\end{cor}
\begin{proof}
This follows from the same proof as Corollary~\ref{cor:killlwithsmallord}. That is, we can eliminate all terms in $\pi^{\ell-\phi(p^m)}p^2\Zp$ modulo $\pi^d$.
\end{proof}

We now turn our attention to case (iii).

\begin{lem}\label{lem:bigterms}
If $(p-2)p^{m-1}< \ell <d-1$, then $\ord_p(a_\ell)\ge 2$.
\end{lem}
\begin{proof}
Since $\ceil[\big]{\frac{\ell}{p^{m-1}}}=p-1$, we have
\[
a_\ell=\binom{(p-1)p^{m-1}}{\ell}=\frac{(p-1)p^{m-1}}{\ell}\prod_{j=1}^{\ell-1}\left(\frac{(p-1)p^{m-1}}{j}-1\right).
\]
All the terms in the product are in $\Zp$ as $\ord_p(j)\le m-1$. Furthermore, since
\[
(p-2)p^{m-1}<\ell <(p-2)p^{m-1}+p^{m-2},
\]
we have $\ord_p(\ell)<m-2$, which tells us that 
\[
\frac{(p-1)p^{m-1}}{\ell}\in p^2\Zp.
\]
Thus, $a_\ell\in p^2\Zp$ as required.
\end{proof}

\begin{cor}\label{cor:mediumterms}
If $(p-2)p^{m-1}< \ell <d-1$, then $a_\ell\pi^{\ell-\phi(p^m)}=O(\pi^d)$.
\end{cor}
\begin{proof}
This follows from the same proof as Corollary~\ref{cor:killlwithsmallord}.
\end{proof}

Finally, we study case (iv), which will play a key role in our algorithm of finding explicit uniformizers.

\begin{prop}\label{prop:d-1}
We have $a_{d-1}\pi^{d-1-\phi(p^m)}=\pi^{d-1}+O(\pi^{d})$.
\end{prop}
\begin{proof}
As in the proof of Lemma~\ref{lem:bigterms}, we have $\ceil[\big]{\frac{d-1}{p^{m-1}}}=p-1$. Thus,
\begin{equation}\label{eq:expandd-1}
a_{d-1}=\binom{(p-1)p^{m-1}}{d-1}=\frac{(p-1)p^{m-1}}{d-1}\prod_{j=1}^{d-2}\left(\frac{(p-1)p^{m-1}}{j}-1\right).
\end{equation}
Note that
\begin{equation}\label{eq:1stproduct}
\frac{(p-1)p^{m-1}}{d-1}=\frac{(p-1)p^{m-1}}{(p-2)p^{m-1}+p^{m-2}}=\frac{(p-1)p}{(p-2)p+1}{= -p\mod p^2\Zp}.
\end{equation}

It remains to study the product in \eqref{eq:expandd-1}. Modulo $p$, the only terms that are not $-1$ are exactly when $j=ip^{m-1}$, $i=1,2,\ldots,p-2$. {We have
\begin{align*}
\prod_{j=1}^{d-2}\left(\frac{(p-1)p^{m-1}}{j}-1\right)&=\prod_{i=1}^{p-2}\left(\frac{p-1}{i}-1\right)\mod p\Zp\\
&=-\prod_{i=1}^{p-2}\left(\frac{1}{i}+1\right)\mod p\Zp\\
&=-\prod_{i=1}^{p-2}\left(i+1\right)\mod p\Zp\\
&=-(p-1)!\mod p\Zp\\
&=1\mod p\Zp
\end{align*}}
by Wilson's theorem. On multiplying this with \eqref{eq:1stproduct}, we deduce that
\[
a_{d-1}{=-p\mod p^2\Zp.}
\]

We can now conclude that
\[
a_{d-1}\pi^{d-1-\phi(p^m)}=\frac{-p}{\pi^{\phi(p^m)}}\pi^{d-1}+O(\pi^d)=\pi^{d-1}+O(\pi^d)
\]
because $\frac{-p}{\pi^{\phi(p^m)}}=1+O(\pi)$ by \eqref{eq:expand} as $v(a_\ell\pi^{-\phi(p^m)})\ge0$ for all $\ell$.
\end{proof}

\section{Explicit uniformizers}\label{sec:uniformizer}

We now explain how to construct an explicit uniformizer of $K_{2,1}$. On combining Lemma~\ref{lem:largeterms}, Corollaries~\ref{cor:killlwithsmallord},  \ref{cor:termslargeval}, \ref{cor:mediumterms} and Proposition~\ref{prop:d-1} and setting $m=2$, we deduce from \eqref{eq:expand} that
\begin{equation}\label{eq:newexpand}
      \frac{-p}{\pi^{\phi(p^2)}} = 1+\sum_{t=1}^{p-2} \frac{(-1)^t}{t+1}p\pi^{(t-p+1)p}+\pi^{(p-2)p+1}+O(\pi^{(p-2)p+2}).
\end{equation}
This now allows us to write down an explicit uniformizer of $K_{2,1}$:

\begin{thm}\label{thm:1}
The expression
\[
\pi^{-2p+3}\left(\sqrt[p]{p}+\pi^{p-1}+\sum_{t=1}^{p-2}\frac{(-1)^t}{t+1}{\sqrt[p]{p}}\pi^{t}\right)
\]
gives a uniformizer of $\Qp(\zeta_{p^2},\sqrt[p]{p})$.
\end{thm}
\begin{proof}
Note that
\[
\begin{split}
\left(\frac{-\sqrt[p]{p}}{\pi^{p-1}}-1-\sum_{t=1}^{p-2}\frac{(-1)^t}{t+1}{\sqrt[p]{p}}\pi^{t-p+1}\right)^p=- \frac{p}{\pi^{\phi(p^2)}} - 1- \sum_{t=1}^{p-2} \frac{(-1)^t}{t+1}p\pi^{(t-p+1)p}\\+O( \pi^{\phi(p^2)})
\end{split}
\]
since all the summands inside the parentheses on the left-hand side have non-negative valuations, which tells us that all  cross terms of the expansion are in {$p\Zp[\pi]=\pi^{\phi(p^2)}\Zp[\pi]$}. If we combine this with \eqref{eq:newexpand}, we deduce that
\[
\left(\frac{-\sqrt[p]{p}}{\pi^{p-1}}-1-\sum_{t=1}^{p-2}\frac{(-1)^t}{t+1}{\sqrt[p]{p}}\pi^{t-p+1}\right)^p=\pi^{(p-2)p+1}+O(\pi^{(p-2)p+1}).
\]
Thus, on dividing both sides by $\pi^{(p-2)p}$, we have
\[
-\left(\pi^{-2p+3}\left(\sqrt[p]{p}+\pi^{p-1}+\sum_{t=1}^{p-2}\frac{(-1)^t}{t+1}{\sqrt[p]{p}}\pi^{t}\right)\right)^p=\pi+O(\pi^{2}).
\]
Thus,
\[
v\left(\pi^{-2p+3}\left(\sqrt[p]{p}+\pi^{p-1}+\sum_{t=1}^{p-2}\frac{(-1)^t}{t+1}{\sqrt[p]{p}}\pi^{t}\right)\right)=\frac{1}{p}
\]
as required.
\end{proof}

The fact that the $\pi$-adic valuation of $\sqrt[p]{p}$ is  $p-1$ allows us to slightly modify the uniformizer in Theorem~\ref{thm:1} on removing the constant $\sqrt[p]{p}$ from the summands, resulting in a sum that can be expressed solely in terms of $\pi$. To do this, notice that \eqref{eq:newexpand} can be rewritten as
\beq\label{eq:newexpand2}
\frac{-p}{\pi^{\phi(p^2)}}=1+\sum_{t=1}^{p-2}\frac{(-1)^{t+1}}{t+1}\left( \frac{-p}{\pi^{\phi(p^2)}} \right)\pi^{tp} + \pi^{(p-2)p+1} + O(\pi^{(p-2)p+2}).
\eeq
Therefore, if we multiply both sides by $\pi^{(p-2)p}$, we get
\[
\frac{-p}{\pi^{\phi(p^2)}}\cdot\pi^{(p-2)p}=\pi^{(p-2)p}+O(\pi^{(p-2)p+2}),
\]
which allows us to replace the last term of the sum in \eqref{eq:newexpand2} by $\pi^{(p-2)p}$. We can then play the same game and multiply both sides of the new equation by $\pi^{(p-3)p}$ to remove the constant $p$ from the second last term of the sum. If we continue this way, we can proceed by a backward recurrence relation to remove the constant $p$ from  one summand at a time.

We now explain how this is done concretely. Let $b_t = \frac{(-1)^{t+1}}{t+1}$ and $A_t = \frac{-p}{\pi^{\phi(p^2)}}\pi^{tp}$, which allows to rewrite \eqref{eq:newexpand2} in a more compact form:
\beq\label{eq:compexpand}
\frac{-p}{\pi^{\phi(p^2)}}=1+\sum_{t=1}^{p-2}b_tA_t + \pi^{(p-2)p+1} + O(\pi^{(p-2)p+2}).
\eeq
Multiplying both sides of this equation by $\pi^{kp}$ ($1 \leq k \leq p-2$) and using the fact that $A_t \cdot \pi^{kp} = A_{t+k}$, we get
\begin{equation}\label{eq:relAt}
A_k = \frac{-p}{\pi^{\phi(p^2)}} \cdot \pi^{kp} = \pi^{kp}+\sum_{t=1}^{p-2-k}b_tA_{k+t} + O(\pi^{(p-2)p+2}),   
\end{equation}
which allows us to express $A_k$ in terms of $A_{k+1},\dots,A_{p-2}$. Therefore, starting from ${A_{p-2} \stackrel{(*)}{=} \pi^{(p-2)p}+O(\pi^{(p-2)p+2})}$, we can obtain all the other $A_k$ by a backward recurrence relation given by  the following lemma.

\begin{lem}\label{lem:formulaAt}
Define
\[
B_0 = 1, \quad B_n = \sum_{k=1}^n b_kB_{n-k}, \quad k=1,\dots,p-2.
\]
Then
\[
A_t = \sum_{k=0}^{p-2-t} B_k\pi^{(t+k)p} + O(\pi^{(p-2)p+2}).
\]
\end{lem}
\begin{proof}
We proceed by (descending) induction. The base case is simply $(*)$ stated above. Suppose now that the lemma holds for $t,t+1,\dots,p-2$. By \eqref{eq:relAt}, we have
\begin{align*}
A_{t-1} &= \pi^{(t-1)p} + \sum_{s=1}^{p-2-(t-1)}b_sA_{t-1+s} + O(\pi^{(p-2)p+2}) \\
&= \pi^{(t-1)p} + \sum_{s=1}^{p-2-(t-1)}b_s\left( \sum_{k=0}^{p-2-(t-1+s)}B_k\pi^{(t-1+s+k)p} \right) + O(\pi^{(p-2)p+2}) \\
&= \pi^{(t-1)p} + \sum_{s=1}^{p-2-(t-1)}b_s\left( \sum_{k=s}^{p-2-(t-1)}B_{k-s}\pi^{(t-1+k)p} \right) + O(\pi^{(p-2)p+2}) \\
&= B_0\pi^{(t-1)p} + \sum_{k=1}^{p-2-(t-1)} \left( \sum_{s=1}^{k}b_sB_{k-s} \right)\pi^{(t-1+k)p} + O(\pi^{(p-2)p+2}) \\
&= \sum_{k=0}^{p-2-(t-1)} B_k\pi^{(t-1+k)p} + O(\pi^{(p-2)p+2}),
\end{align*}
so the lemma holds for $t-1$ as well.
\end{proof}

With Lemma~\ref{lem:formulaAt} in hand, it is now possible to rewrite \eqref{eq:newexpand2} as follows.

\begin{lem}\label{lem:B}
We have
\[
\frac{-p}{\pi^{\phi(p^2)}} = \sum_{t=0}^{p-2} B_t \pi^{tp}+\pi^{(p-2)p+1}+O(\pi^{(p-2)p+2}).
\]
\end{lem}
\begin{proof}
Combining \eqref{eq:compexpand} and Lemma \ref{lem:formulaAt}, we have
\begin{align*}
\frac{-p}{\pi^{\phi(p^2)}} &= 1+\sum_{t=1}^{p-2}b_t\left( \sum_{k=0}^{p-2-t} B_k\pi^{(t+k)p} \right) + \pi^{(p-2)p+1} + O(\pi^{(p-2)p+2}) \\
&= 1+\sum_{t=1}^{p-2}b_t \sum_{k=t}^{p-2} B_{k-t}\pi^{kp} + \pi^{(p-2)p+1} + O(\pi^{(p-2)p+2}) \\
&= B_0+\sum_{k=1}^{p-2} \left( \sum_{t=1}^{k}b_tB_{k-t} \right)\pi^{kp} + \pi^{(p-2)p+1} + O(\pi^{(p-2)p+2}) \\
&= \sum_{k=0}^{p-2} B_k\pi^{kp} + \pi^{(p-2)p+1} + O(\pi^{(p-2)p+2}).
\end{align*}
\end{proof}

\begin{thm}
The expression
\[
\pi^{-2p+3}\left(\sqrt[p]{p}+\sum_{t=0}^{p-2}B_t\pi^{t+p-1}\right)
\]
gives a uniformizer of $\Qp(\zeta_{p^2},\sqrt[p]{p})$.
\end{thm}
\begin{proof}
It follows from exactly the same proof as Theorem~\ref{thm:1} on replacing \eqref{eq:newexpand} by Lemma~\ref{lem:B}.
\end{proof}

We have worked out this uniformizer explicitly for small values of $p$:
\begin{eqnarray*}
   &p=3: \quad &\pi^{-3}(3^{\frac{1}{3}}+\pi^2+2\pi^3) \\
&p=5: \quad& \pi^{-7}(5^{\frac{1}{5}}+\pi^4+3\pi^5+2\pi^6+4\pi^7)\\
&p=7: \quad& \pi^{-11}(7^{\frac{1}{7}}+\pi^6+4\pi^7+4\pi^8+5\pi^9+5\pi^{10}+4\pi^{11}) \\
&p=11: \quad& \pi^{-19}(11^{\frac{1}{11}}+\pi^{10}+6\pi^{11}+10\pi^{12}+6\pi^{13}+5\pi^{14}+6\pi^{15}+3\pi^{16}+\\
&&7\pi^{18}+9\pi^{19}) \\
&p=13: \quad& \pi^{-23}(13^{\frac{1}{13}}+\pi^{12}+7\pi^{13}+\pi^{14}+6\pi^{15}+4\pi^{16}+4\pi^{17}+2\pi^{18}+\\
&&11\pi^{19}+7\pi^{20}+10\pi^{21}+4\pi^{22}+\pi^{23}) 
\end{eqnarray*}

\begin{remark}\label{rk:fail}
 The key of the construction relies crucially on finding an integer  $\ell$ such that $p|| a_{\ell}$ in the expression \eqref{eq:expand} with  $p\nmid \ell$. In the case $m=2$, our calculations in \S\ref{sec:pre} show that $\ell=d-1=(p-2)p+1$ satisfies these two conditions.  Unfortunately, when $m>2$, our choice of $d$ satisfies the first condition, but not the second one. In all the numerical examples where  $m>2$ we have studied, we have not found  a single instance where an integer $\ell$ satisfying  both conditions exists. It seems to suggest that completely new ideas will be required to construct explicit uniformizers of $K_{m,n}$ for arbitrary $m$ and $n$.
\end{remark}


\end{document}